\documentclass[a4paper, 12pt]{article}
\pdfoutput=1

\usepackage[T1]{fontenc}
\usepackage[utf8]{inputenc}

\usepackage{amsfonts,amsmath,amssymb,mathtools,dsfont,microtype}
\usepackage[dvipsnames]{xcolor}
\usepackage[noBBpl]{mathpazo}

\DeclareFontFamily{U} {cmr}{}

\DeclareFontShape{U}{cmr}{m}{n}{
    <-6> cmr5
    <6-7> cmr6
    <7-8> cmr7
    <8-9> cmr8
    <9-10> cmr9
    <10-12> cmr10
    <12-> cmr12}{}

\DeclareSymbolFont{Xcmr} {U} {cmr}{m}{n}
\DeclareMathSymbol{\Delta}{\mathord}{Xcmr}{'001}
\DeclareMathSymbol{\Upsilon}{\mathord}{Xcmr}{'007}
\DeclareMathSymbol{\Omega}{\mathord}{Xcmr}{'012}
\DeclareMathSymbol{\Gamma}{\mathord}{Xcmr}{'000}
\DeclareMathSymbol{\Theta}{\mathord}{Xcmr}{'002}
\DeclareMathSymbol{\Lambda}{\mathord}{Xcmr}{'003}
\DeclareMathSymbol{\Xi}{\mathord}{Xcmr}{'004}
\DeclareMathSymbol{\Pi}{\mathord}{Xcmr}{'005}
\DeclareMathSymbol{\Sigma}{\mathord}{Xcmr}{'006}
\DeclareMathSymbol{\Phi}{\mathord}{Xcmr}{'010}
\DeclareMathSymbol{\Psi}{\mathord}{Xcmr}{'011}

\usepackage[
    labelsep = period, labelfont = bf, justification = centering
]{caption}
\usepackage{float,graphicx,subcaption}
\DeclareCaptionSubType*[roman]{figure}

\usepackage{booktabs,makecell}
\usepackage{enumitem,mdwlist}
\setlist[itemize]{topsep=0ex,itemsep=0ex,parsep=0.4ex}
\setlist[enumerate]{topsep=0ex,itemsep=0ex,parsep=0.4ex}

\usepackage{pgf,tikz,ifthen,calc}
\usepackage{tkz-euclide}
\tkzSetUpPoint[fill=black, size = 3pt]
\tkzSetUpLine[color=black, line width=0.6pt]
\tkzSetUpCircle[color=black, line width=0.6pt]

\usepackage[
    left = 2.5cm, right = 2.5cm, top = 2.5cm, bottom = 2.5cm, headsep = 12pt,
    headheight = 15pt
]{geometry}
\usepackage{parskip}

\usepackage[hyphens]{url}
\usepackage[
    linktoc = all, hidelinks, colorlinks, unicode=true, pagebackref=true
]{hyperref}
\usepackage[capitalise, compress, nameinlink, noabbrev]{cleveref}

\hypersetup{
    linkcolor={blue!70!black}, citecolor={green!70!black},
    urlcolor={blue!70!black}
}

\usepackage{amsthm}
\theoremstyle{plain}
\newtheorem{thm}{Theorem}[section]

\newtheorem{cor}[thm]{Corollary}
\newtheorem{conj}[thm]{Conjecture}
\newtheorem{lem}[thm]{Lemma}
\newtheorem{obs}[thm]{Observation}

\newtheorem{que}[thm]{Question}

\theoremstyle{definition}
\newtheorem{defi}[thm]{Definition}

\crefname{subsection}{\S}{\S\S}

\DeclareFontFamily{U}{matha}{\hyphenchar\font45}
\DeclareFontShape{U}{matha}{m}{n}{
    <5> <6> <7> <8> <9> <10> gen * matha
    <10.95> matha10 <12> <14.4> <17.28> <20.74> <24.88> matha12
}{}
\DeclareSymbolFont{matha}{U}{matha}{m}{n}

\DeclareMathSymbol{\specialuparrow}{\mathrel}{matha}{"D2}
\DeclareMathSymbol{\specialrightarrow}{\mathrel}{matha}{"D1}

\renewcommand*{\backref}[1]{}
\renewcommand*{\backrefalt}[4]{
    \ifcase #1 Not cited.\or $\specialuparrow$#2\else $\specialuparrow$#2\fi}

\renewcommand{\epsilon}{\varepsilon}

\renewcommand{\geq}{\geqslant}
\renewcommand{\leq}{\leqslant}

\newcommand{\cc}{\texttt{c}}
\newcommand{\defn}[1]{\textcolor{Maroon}{\emph{#1}}}

\title{Multicolor Ramsey numbers\\
    of ordered matchings}
\author{Xiangyu Li\footnotemark[1]}
\date{}

\begin{document}

\maketitle

\renewcommand{\thefootnote}{\fnsymbol{footnote}}\urlstyle{same}

\footnotetext[1]{Department of Mathematics, University of Toronto, Canada.
\emph{Email}: \href{mailto:xiangyuu.li@mail.utoronto.ca}
{\nolinkurl{xiangyuu.li@mail.utoronto.ca}}. Research supported by a University
of Toronto Excellence Award (UTEA).}

\renewcommand{\thefootnote}{\arabic{footnote}}

\begin{abstract}
    For an ordered graph \(H\) and an integer \(q\geq2\), let \(r_{<}(H;q)\)
    denote the \(q\)-color ordered Ramsey number of \(H\). Conlon, Fox, Lee and
    Sudakov asked whether, for every \(q\geq3\), there is a constant \(c_q\)
    such that \(r_{<}(M;q)\leq n^{c_q\log n}\) for every ordered matching \(M\)
    on \(n\) vertices. We answer this negatively in a strong form: for every
    \(q\geq2\), there is \(c_q>0\) such that almost every perfect matching \(M\)
    on \([n]\) satisfies
    \[
        r_{<}(M;q)>2^{c_q(\log n)^q/(\log\log n)^{q-1}}.
    \]
    This matches the general upper bound up to a factor of
    \((\log\log n)^{q-1}\) in the exponent. We also give two applications.
    First, we prove a lower bound conjectured by Fox, He and Wigderson for
    multicolor Ramsey numbers of acyclic digraphs. Second, we strengthen a
    result of Axenovich, Rollin and Ueckerdt by giving a superquasipolynomial
    lower bound on the maximum chromatic number of \(M\)-free ordered graphs,
    for almost every perfect matching \(M\) on \([n]\).
\end{abstract}
\section{Introduction}

An \defn{ordered graph} is a graph equipped with a linear order on its vertices,
and copies are required to preserve this order. For an ordered graph \(H\) and
an integer \(q\geq2\), the \(q\)-color \defn{ordered Ramsey number}
\(r_{<}(H;q)\) is the least \(N\) such that every \(q\)-edge-coloring of the
complete ordered graph on \(N\) vertices contains a monochromatic copy of \(H\).
The systematic study of these numbers was initiated independently by Balko,
Cibulka, Kr\'al and Kyn\v{c}l~\cite{BCKK} and by Conlon, Fox, Lee and
Sudakov~\cite{CFLS}. See also Balko's survey~\cite{BalkoSurvey} for further
results and open questions.

In this paper, we focus on the multicolor Ramsey numbers of ordered matchings.
The best known upper bound for any ordered matching \(M\) on \(n\) vertices is
\(r_{<}(M;q)\leq 2^{O_q((\log n)^q)}\)~\cite{CFLS,BalkoGrinerova},\footnote{Throughout
the paper, $\log$ denotes the base-$2$ logarithm.} whereas the best previous
lower bound is \(r_{<}(M;q)\geq 2^{\Omega_q((\log n)^2/\log\log n)}\) for almost
every perfect matching \(M\) on \([n]\)~\cite{CFLS,BalkoGrinerova}. Conlon et~al.\ suggested that the lower bound is closer to the truth (see
also~\cite[Problem~10]{BalkoSurvey}).

\begin{que}[{\cite[Problem~6.8]{CFLS}}]
    \label{prob:cfls}
    Show that for any natural number \(q\geq3\) there exists a constant \(c_q\)
    such that \(r_{<}(M;q)\leq n^{c_q\log n}\) for any matching \(M\) on \(n\)
    vertices.
\end{que}

Our main result answers this negatively for every \(q\geq3\) and matches the
general upper bound up to a factor of \((\log\log n)^{q-1}\) in the exponent.

\begin{thm}\label{thm:main}
    For every \(q\geq2\), there is \(c_q>0\) such that almost every perfect
    matching \(M\) on \([n]\) satisfies
    \[
        r_{<}(M;q)>2^{c_q(\log n)^q/(\log\log n)^{q-1}}.
    \]
\end{thm}

Theorem~\ref{thm:main} also applies to \defn{oriented Ramsey numbers}. For an
acyclic digraph \(H\), let \(\overrightarrow r_q(H)\) be the least \(N\) such
that every \(q\)-edge-colored tournament on \(N\) vertices contains a
monochromatic copy of \(H\). For integers \(q,\Delta\geq1\) and \(n>\Delta\),
let \(H_{q,n,\Delta}\) be an \(n\)-vertex acyclic digraph of maximum total
degree at most \(\Delta\) maximizing \(\overrightarrow r_q(H)\). Fox, He and
Wigderson~\cite{FHW} proved that, for every fixed \(q\geq2\) and
\(\Delta\geq3\),
\[
    \Omega\!\left(\frac{(\log n)^2}{\log\log n}\right)
    \leq \log \overrightarrow r_q(H_{q,n,\Delta})
    \leq O_{q,\Delta}\!\left((\log n)^{2^{2q-1}}\right).
\]
Equivalently, there exists an \(n\)-vertex acyclic digraph \(H\) of maximum
total degree at most \(\Delta\) satisfying the lower bound, while every such
\(H\) satisfies the upper bound. In particular, the power of \(\log n\) is
independent of \(q\) in the lower bound but grows exponentially with \(q\) in
the upper bound. They conjectured that the lower-bound exponent can tend to
infinity with \(q\), while the upper-bound exponent grows subexponentially.

\begin{conj}[{\cite[Conjecture~6.3]{FHW}}]
    \label{conj:fhw}
    There exist \(\Delta\geq3\) and constants \(c_q=\omega_q(1)\) and
    \(C_q=2^{o(q)}\) such that
    \[
        \Omega_q(\log^{c_q} n)
        \leq \log \overrightarrow r_q(H_{q,n,\Delta})
        \leq O_q(\log^{C_q} n).
    \]
\end{conj}

We prove the lower-bound part of this conjecture with \(\Delta=3\) and exponent
\(q-o(1)\). The proof uses a standard reduction that transfers ordered Ramsey
lower bounds to the oriented setting (see~\cite[Lemma~5.2]{FHW}).

\begin{cor}\label{cor:oriented}
    For every fixed \(q\geq2\), as \(n\to\infty\),
    \[
        \log \overrightarrow r_q(H_{q,n,3})
        =\Omega_q\!\left(
            \frac{(\log n)^q}{(\log\log n)^{q-1}}
        \right).
    \]
\end{cor}

\begin{proof}
    We may assume that \(n\) is even.\footnote{For odd \(n\), add an isolated
    vertex to the construction on \(n-1\) vertices and adjust constants as
    needed.} By Theorem~\ref{thm:main}, choose an ordered matching \(M\) on
    \([n]\) such that
    \[
        r_{<}(M;q)>
        \left\lfloor 2^{c_q(\log n)^q/(\log\log n)^{q-1}}\right\rfloor
        =:N
    \]
    for some \(c_q>0\). Define \(H\) by adding the edges \(i(i+1)\),
    \(i\in[n-1]\), to \(M\) and orienting every edge forwards. Then \(H\) is
    acyclic and has maximum total degree at most \(3\).

    Take a \(q\)-coloring of the complete ordered graph on \([N]\) with no
    monochromatic ordered copy of \(M\). Orient every edge forwards to obtain a
    \(q\)-colored tournament. Since \(H\) contains the directed path
    \(1\to2\to\cdots\to n\), every embedding of \(H\) into this tournament
    preserves the vertex order. A monochromatic copy of \(H\) would therefore
    give a monochromatic ordered copy of \(M\), a contradiction. Thus
    \[
        \overrightarrow r_q(H_{q,n,3})\geq
        \overrightarrow r_q(H)>N. \qedhere
    \]
\end{proof}

We give one last application of Theorem~\ref{thm:main}. For all ordered graphs
$H$, define
\[
    f_{<}(H)=\sup\{\chi(G):G\text{ is an }H\text{-free ordered graph}\}.
\]
A beautiful recent result of Bourneuf, Cocquet, Tang and Thomass\'e~\cite{BCTT}
implies that $f_{<}(M)<\infty$ for every ordered matching $M$. Thus, one can ask
the following question (see~\cite[Question~3]{ARU} for a more general
version).
\begin{que}\label{q:matching}
    Determine the asymptotic growth, as $n\to\infty$, of
    \[
        f_<(n) = \max\{f_{<}(M):
            M\text{ is an ordered matching on }n\text{ vertices}\}.
    \]
\end{que}
To this end, Axenovich, Rollin and Ueckerdt~\cite{ARU} showed
that, for all sufficiently large even $n$, some $n$-vertex ordered matching $M$
satisfies $f_{<}(M)\geq 2^{\Omega((\log n)^2/\log\log n)}$. We strengthen their
bound by adapting our proof of Theorem~\ref{thm:main}.

\begin{thm}\label{thm:chromatic}
    There exists \(c>0\) such that almost every perfect matching \(M\) on
    \([n]\) satisfies
    \[
        f_{<}(M)\geq 2^{2^{c\sqrt{\log n}}}.
    \]
\end{thm}

\section{Proofs of Theorems~\ref{thm:main} and~\ref{thm:chromatic}}
\begin{defi}
    A \(K_t\) \defn{interval minor} in an ordered graph \(G\) is a collection of
    nonempty intervals \(I_1<\cdots<I_t\) such that there is an edge between
    \(I_i\) and \(I_j\) for every \(i<j\). Let \(\cc(G)\) denote the largest
    \(t\) for which \(G\) contains \(K_t\) as an interval minor, with
    \(\cc(\varnothing)=0\).
\end{defi}

In our proof, we focus on $K_t$-interval minors rather than individual
matchings. Indeed, it is known that almost every ordered perfect matching $M$ on
$n$ vertices satisfies $\cc(M)>n^{1/4}$. Thus, a
$q$-coloring of $K_N$ in which every color class $G$ satisfies
$\cc(G)\leq n^{1/4}$ contains no monochromatic copy of any such $M$.

\begin{defi}
    For ordered graphs \(G\) and \(H\), the \defn{ordered substitution of \(H\)
    into \(G\)}, denoted by \(G[H]\), is obtained by replacing each vertex
    \(u\in V(G)\) with a consecutive ordered copy \(B_u\) of \(H\). The blocks
    \(B_u\) occur in the order of their corresponding vertices. For distinct
    \(u,v\in V(G)\), \(B_u\) is complete to \(B_v\) if \(uv\in E(G)\), and
    anticomplete otherwise.

    For \(d\geq1\), define the \defn{\(d\)-fold iterated substitution
    \(G^{[d]}\)} recursively by \(G^{[1]}=G\) and \(G^{[d]}=G[G^{[d-1]}]\) for
    \(d\geq2\). Its vertices can be identified with \(V(G)^d\), ordered
    lexicographically, and two distinct tuples are adjacent if and only if their
    entries in the first coordinate in which they differ are adjacent in \(G\).
\end{defi}

Our first lemma describes the behavior of complete interval minors under ordered
substitutions. A similar argument appears in~\cite{BCTT}.

\begin{lem}\label{lem:substitution}
    For all ordered graphs \(G\) and \(H\),
    \[
        \cc(G[H])\leq2\cc(G)+\omega(G)\cc(H).
    \]
    In particular, for every integer \(d\geq1\),
    \(\cc(G^{[d]})\leq(\omega(G)+2)^{d-1}\cc(G)\).
\end{lem}

\begin{proof}
    Let \(H_u\) be the copy of \(H\) replacing \(u\in V(G)\), and fix a
    \(K_t\)-interval minor \(I_1<\cdots<I_t\) in \(G[H]\). Say that \(I_i\) is
    \defn{internal} if \(I_i\subseteq V(H_u)\) for some \(u\in V(G)\), and
    \defn{crossing} otherwise. The vertices \(u\) for which \(H_u\) contains an
    internal interval form a clique in \(G\), and each \(H_u\) contains at most
    \(\cc(H)\) internal intervals. Thus, at most \(\omega(G)\cc(H)\) of the
    intervals \(I_i\) are internal.

    Write the crossing intervals as \(I_{n_1}<\cdots<I_{n_p}\), and let
    \(J_i=\{u\in V(G):V(H_u)\cap I_{n_i}\neq\varnothing\}\). These are intervals
    of \(V(G)\), each containing at least two vertices, and satisfy
    \[
        \max J_i\leq\min J_{i+1}<\max J_{i+1}\leq\min J_{i+2}
    \]
    for \(1\leq i\leq p-2\). Therefore, the odd- and even-indexed \(J_i\) form
    complete interval minors in \(G\). It follows that \(p\leq2\cc(G)\) and
    hence \(t\leq2\cc(G)+\omega(G)\cc(H)\). The bound for \(G^{[d]}\) follows by
    induction on \(d\).
\end{proof}

\begin{defi}
    For integers \(q\geq k\geq1\), a \defn{\(k\)-cover} of \(K_N\) is a sequence
    \(G_1,\ldots,G_q\) of spanning subgraphs of \(K_N\) such that every edge
    belongs to exactly \(k\) of them.
\end{defi}

In other words, a \(k\)-cover of \(K_N\) is a coloring of \(E(K_N)\) where every
edge receives a set of \(k\) colors out of a total of \(q\) colors.

\begin{obs}\label{obs:cover-substitution}
    If \(G_1,\ldots,G_q\) is a \(k\)-cover of \(K_N\), then
    \(G_1^{[d]},\ldots,G_q^{[d]}\) is a \(k\)-cover of \(K_{N^d}\).
\end{obs}
\begin{lem}\label{lem:thinning}
    Let \(q\geq k\geq2\) and \(N\geq2\). Every \(k\)-cover \(G_1,\ldots,G_q\) of
    \(K_N\) has spanning subgraphs \(G_i'\subseteq G_i\) forming a
    \((k-1)\)-cover such that \(\omega(G_i')<\lceil4q\log N\rceil\) for every
    \(i\in[q]\).
\end{lem}

\begin{proof}
    Let \(r=\lceil4q\log N\rceil\). Independently for each edge of \(K_N\),
    delete it from one of the \(k\) graphs containing it, chosen uniformly at
    random. This gives a \((k-1)\)-cover. For each \(i\in[q]\), a fixed
    \(r\)-set is a clique in \(G_i'\) with probability at most
    \((1-1/q)^{\binom r2}\leq 2^{-\binom r2/q}\). Therefore, by the union bound,
    \[
        \Pr\bigl(\omega(G_i')\geq r\text{ for some }i\in[q]\bigr)
        \leq q\binom Nr 2^{-\binom r2/q}
        \leq qN^{-r/2}
        \leq q2^{-r/2}
        <q2^{-q}<1,
    \]
    where we used $r-1\geq3q\log N$ and $r>2q$. Thus, some choice of deletions
    ensures that $\omega(G_i')<r$ for every $i\in[q]$.
\end{proof}
Combining Lemma~\ref{lem:substitution}, Observation~\ref{obs:cover-substitution}
and Lemma~\ref{lem:thinning}, we have the following.
\begin{lem}\label{lem:amplification}
    Let \(q\geq k\geq2\), \(N\geq2\), and \(d\geq1\) be integers. Every
    \(k\)-cover \(G_1,\ldots,G_q\) of \(K_N\) gives a \((k-1)\)-cover
    \(H_1,\ldots,H_q\) of \(K_{N^d}\) such that for every \(i \in [q]\),
    \[
        \cc(H_i)\leq(8q\log N)^d\cc(G_i).
    \]
\end{lem}

\begin{proof}
    By Lemma~\ref{lem:thinning}, choose a \((k-1)\)-cover \(G_1',\ldots,G_q'\)
    with \(G_i'\subseteq G_i\) and \(\omega(G_i')<4q\log N\). The graphs
    \(H_i=(G_i')^{[d]}\) form a \((k-1)\)-cover of \(K_{N^d}\) by
    Observation~\ref{obs:cover-substitution}. Lemma~\ref{lem:substitution} then
    gives \(\cc(H_i)\leq(\omega(G_i')+2)^{d-1}\cc(G_i')\) and the result follows
    from \(\omega(G_i')+2\leq8q\log N\) and \(\cc(G_i')\leq\cc(G_i)\).
\end{proof}

\begin{proof}[Proof of Theorem~\ref{thm:main}]
    Fix \(q\geq2\), let \(n\) be sufficiently large and even, and set
    \[
        d:=\lfloor\log n/(4q^2\log\log n)\rfloor.
    \]
    Starting with the \(q\)-cover of \(K_{d^d}\) assigning all \(q\) colors to
    each edge, apply Lemma~\ref{lem:amplification} \(q-1\) times, each time with
    exponent \(d\). The resulting \(1\)-cover is a \(q\)-coloring of
    \(K_{d^{d^q}}\).

    Every application occurs on at most \(d^{d^q}\) vertices, so in
    Lemma~\ref{lem:amplification} we have
    \[
        8q\log N\leq8q\log(d^{d^q})=8qd^q\log d\leq d^{q+1}
    \]
    for sufficiently large \(n\). Thus, every color class \(G_i\) satisfies
    \[
        \cc(G_i)
        \leq d^d(d^{q+1})^{d(q-1)}
        =d^{q^2d}
        \leq n^{1/4},
    \]
    since \(q^2d\log d\leq\frac14\log n\). Finally, by~\cite[Lemma~2.2]{CFLS},
    almost every perfect matching \(M\) on \([n]\) has \(\cc(M)>n^{1/4}\), so
    every \(G_i\) is \(M\)-free. Since \(d=\Theta_q(\log n/\log\log n)\),
    \[
        r_{<}(M;q)>d^{d^q}
        =2^{d^q\log d}
        =2^{\Omega_q\left(
            (\log n)^q/(\log\log n)^{q-1}
        \right)}. \qedhere
    \]
\end{proof}
\begin{proof}[Proof of Theorem~\ref{thm:chromatic}]
    Let \(n\) be a sufficiently large even integer and let
    \(q=\lfloor\sqrt{\log n}/4\rfloor\). Starting with the \(q\)-cover of
    \(K_2\), apply Lemma~\ref{lem:amplification} \(q-1\) times with \(d=2\) to
    get a \(q\)-coloring of \(K_N\), where \(N=2^{2^{q-1}}\). As before, the
    color classes \(G_i\) satisfy
    \[
        \cc(G_i)
        \leq2(8q\log N)^{2(q-1)}
        \leq2^{4q^2}
        \leq n^{1/4}.
    \]
    Again as before, almost every perfect matching \(M\) on \([n]\) satisfies
    \(\cc(M)>n^{1/4}\). Thus, every \(G_i\) is \(M\)-free. Finally, since \(\prod_{i=1}^q\chi(G_i)\geq N\), some \(G_i\)
    has \(\chi(G_i)\geq N^{1/q}\). Therefore
    \[
        f_{<}(M)\geq N^{1/q}
        =2^{2^{q-1}/q}
        =2^{2^{\Omega(\sqrt{\log n})}}.\qedhere
    \]
\end{proof}

\paragraph{Acknowledgment.}
I am very grateful to Lior Gishboliner for his encouragement, careful reading of
the proof, and generous help in improving the exposition of this paper.

\paragraph{AI Disclosure.}
Theorem~\ref{thm:main} was proved without AI assistance. ChatGPT 5.6-Luna
suggested its application to Conjecture~\ref{conj:fhw}, along with a proof
of Corollary~\ref{cor:oriented}. The author rewrote that proof. ChatGPT
also helped polish the exposition. The author takes full responsibility
for all results in this paper.

\fontsize{11pt}{12pt}
\selectfont

\hypersetup{linkcolor={red!70!black}}
\setlength{\parskip}{2pt plus 0.3ex minus 0.3ex}

\bibliographystyle{auxfile.bst}
\bibliography{bib.bib}

@incollection{BalkoSurvey,
  author    = {Balko, Martin},
  title     = {A survey on ordered {R}amsey numbers},
  booktitle = {Sum(m)it280: Surveys in Extremal Combinatorics and Combinatorial Geometry},
  editor    = {Katona, Gyula O. H. and Patk{\'o}s, Bal{\'a}zs and Tompkins, Casey},
  series    = {Bolyai Society Mathematical Studies},
  volume    = {32},
  pages     = {71--96},
  publisher = {Springer},
  address   = {Cham},
  year      = {2026},
  doi       = {10.1007/978-3-032-18810-6_4}
}

@article{BalkoGrinerova,
  author  = {Balko, Martin and Grinerov{\'a}, Kl{\'a}ra},
  title   = {Estimating multicolor ordered {R}amsey numbers},
  journal = {Discrete Mathematics},
  volume  = {349},
  number  = {11},
  pages   = {115286},
  year    = {2026},
  doi     = {10.1016/j.disc.2026.115286}
}

@article{FHW,
  author  = {Fox, Jacob and He, Xiaoyu and Wigderson, Yuval},
  title   = {{R}amsey numbers of sparse digraphs},
  journal = {Israel Journal of Mathematics},
  volume  = {263},
  number  = {1},
  pages   = {1--48},
  year    = {2024},
  doi     = {10.1007/s11856-024-2624-y}
}

@inproceedings{BCTT,
  author    = {Bourneuf, Romain and Cocquet, Julien and Tang, Chaoliang and Thomass{\'e}, St{\'e}phan},
  title     = {A polynomial-time approximation algorithm for complete interval minors},
  booktitle = {Approximation, Randomization, and Combinatorial Optimization. Algorithms and Techniques ({APPROX/RANDOM} 2025)},
  editor    = {Ene, Alina and Chattopadhyay, Eshan},
  series    = {Leibniz International Proceedings in Informatics (LIPIcs)},
  volume    = {353},
  pages     = {15:1--15:23},
  publisher = {Schloss Dagstuhl -- Leibniz-Zentrum f{\"u}r Informatik},
  address   = {Dagstuhl, Germany},
  year      = {2025},
  doi       = {10.4230/LIPIcs.APPROX/RANDOM.2025.15}
}

@article{ARU,
  author  = {Axenovich, Maria and Rollin, Jonathan and Ueckerdt, Torsten},
  title   = {Chromatic number of ordered graphs with forbidden ordered subgraphs},
  journal = {Combinatorica},
  volume  = {38},
  number  = {5},
  pages   = {1021--1043},
  year    = {2018},
  doi     = {10.1007/s00493-017-3593-0}
}

@article{BCKK,
  author  = {Balko, Martin and Cibulka, Josef and Kr{\'a}l, Karel and Kyn{\v{c}}l, Jan},
  title   = {{R}amsey numbers of ordered graphs},
  journal = {Electronic Journal of Combinatorics},
  volume  = {27},
  number  = {1},
  pages   = {P1.16},
  year    = {2020},
  doi     = {10.37236/7816}
}

@article{CFLS,
  author  = {Conlon, David and Fox, Jacob and Lee, Choongbum and Sudakov, Benny},
  title   = {Ordered {R}amsey numbers},
  journal = {Journal of Combinatorial Theory, Series B},
  volume  = {122},
  pages   = {353--383},
  year    = {2017},
  doi     = {10.1016/j.jctb.2016.06.007}
}

\end{document}